\documentclass[10pt]{article}
\usepackage{amssymb,amsmath}
\usepackage{graphicx}
\author{Abdallah Assi\thanks{Universit\'e d'Angers, Math\'ematiques,
49045 Angers cedex 01, France, e-mail:assi@univ-angers.fr}}
\title{Irreducibility criterion for  quasi-ordinary polynomials
\footnote{During the development of this work, the author
visited the Department of Mathematics at the American University
of Beirut, Lebanon. He would like to thank that institution for
hospitality and support. He also would like to think the Center
for Advanced Mathematical Sciences-CAMS for offering access to
many facilities.}}
\date{\mbox{}}

\newtheorem{teorema}{Theorem}[section]

\newtheorem{proposition}[teorema]{Proposition}
\newtheorem{lemma}[teorema]{Lemma}
\newtheorem{definition}[teorema]{Definition}

\newtheorem{theorem}[teorema]{Theorem}

\newenvironment{proof}[1]{\paragraph{\sl Proof#1}}{}
\newenvironment{proofs}[1]{\paragraph{\sl Proof of Theorem 6.1.#1}}{}

\newcommand{\RR}{{\bf R}}
\newcommand{\KK}{{\bf K}}

\newcommand{\NN}{{\bf N}}

\begin{document}

\maketitle

\section*{Introduction}

\noindent Let ${\bf K}$ be an algebraically closed field of
characteristic zero, and let ${\bf R}={\bf
K}[[x_1,\ldots,x_e]]={\bf K}[[\underline{x}]]$ be the ring of
formal power series in $x_1,\ldots,x_e$ over ${\KK}$. Let
$f=y^n+a_1(\underline{x})y^{n-1}+\ldots+a_n({\underline{x}})$ be
 a nonzero  polynomial of ${\bf R}[y]$, and suppose
 that $f$ is irreducible in ${\bf R}[y]$. Suppose that $e=1$ and let $g$ be a nonzero
 polynomial of ${\bf R}[y]$, then define the intersection
 multiplicity of $f$ with $g$, denoted int$(f,g)$, to be the $x$-order of the $y$ resultant
 of $f$ and $g$. The set of int$(f,g), g\in {\bf R}[y]$, defines a
 semigroup, denoted $\Gamma(f)$. It is will known that a set of
 generators of $\Gamma(f)$ can be computed from polynomials having
 the maximal contact with $f$ (see [1] and [6]), namely, there exist
 $g_1,\ldots,g_h$ such that $n, {\rm int}(f,g_1),\ldots,{\rm int}(f,g_h)$ generate
 $\Gamma(f)$ and for all $1\leq k\leq h$, the Newton-Puiseux expansion of $g_k$ coincides with that of $f$
 until a characteristic exponent of $f$. In [1], Abhyankar
 introduced a special set of polynomials called the approximate
 roots of $f$. These polynomials have the advantage that they can
 be calculated from the equation of $f$ by using the Tschirnhausen
 transform.  Suppose that $e\geq 2$ and that the discriminant of $f$ is of the form $x_1^{N_1}.\ldots.x_e^{N_e}.u(x_1,\ldots,x_e)$, where
 $u$ is a unit in ${\bf K}[[\underline{x}]]$
 (such a polynomial is called quasi-ordinary polynomial).
 By Abhyankar-Jung Theorem, the roots of $f(x_1,\ldots,x_e,y)=0$ are all in
 ${\bf K}[[x_1^{1\over n},\ldots,x_e^{1\over n}]]$, in particular there exists a power series
 $y(t_1,\ldots,t_e)=\sum_pc_pt_1^{p_1}.\ldots.t_e^{p_e}\in {\bf K}[[t_1,\ldots,t_e]]$ such
 that $f(t_1^n,\ldots,t_e^n,y(t_1,\ldots,t_e))=0$ and the other roots of $f(t_1^n,\ldots,t_e^n,y)=0$ are
 the conjugates of $y(t_1,\ldots,t_e)$ with respect to the $n$th roots of unity in ${\bf K}$. Given a polynomial $g$ of ${\bf R}[y]$, we
 define the order of $g$ to be the leading exponent with respect to the lexicographical order
  of the smallest homogeneous component of
$g(t_1^n,\ldots,t_e^n,y(t_1,\ldots,t_e))$. The set of orders of
polynomials of ${\bf R}[y]$  defines  a semigroup. In this paper
we first prove that the canonical basis of $(n{\bf Z})^e$ with the
set of orders of  the approximate roots of $f$ generate the
semigroup of $f$, then we give, using these approximate roots and
the notion of generalized Newton polygons, a criterion for a
quasi-ordinary polynomial to be irreducible.  Note that if $e=1$,
then $f$ is quasi-ordinary, in particular our results generalize
those of Abhyankar (see [1] and [3]).

 \noindent The paper is organized as follows: in Section 1 we
 introduce the notion of approximate roots of a polynomial in one
variable  over a commutative ring with unity. In Section 2 we show
how to
 associate a semigroup with an irreducible quasi-ordinary
 polynomial of  ${\bf R}[y]$. In Section 3 we introduce the notion of pseudo roots of
 a quasi-ordinary polynomial $f$ then we prove that the orders of these polynomials together with the canonical basis of
 $(n{\bf Z})^e$ give a set of generators of the semigroup of $f$. This result remains true if we
 replace the pseudo roots of $f$ by its set of approximate roots. This is what we prove in Section 4.
 Sections 5 and 6 are devoted to the irreducibility criterion: in Section 5 we introduce the notion of generalized
 Newton polygon, and we define the notion of straightness of a polynomial with respect to a set
 of polynomials, then we use these notions in section 6
in order to decide if a given quasi-ordinary polynomial is
irreducible. We end the paper with some examples in section 7.

\section {G-adic expansions}

\noindent Let $\RR$ be a commutative ring with unity and let
$\RR[y]$ be the ring of polynomials in $y$ with coefficients in
$\RR$. Let $f=y^n+a_1y^{n-1}+\ldots+a_n$ be a monic polynomial of
$\RR[y]$ of  degree $n > 0$ in $y$. Let $d\in{\NN}$ and suppose
that $d$ divides $n$. Let $g$ be a monic polynomial in $\RR[y]$ of
degree $\displaystyle{n\over d}$ in $y$. There exist unique
polynomials $a_1(y),\ldots,a_d(y)\in\RR[y]$ such that:

$$
 f=g^d+\sum_{i=1}^{d}a_i(y).g^{d-i}
$$

\noindent and for all $1\leq i\leq d$, if we denote by deg$_y$ the
$y$-degree, then deg$_y(a_i)< \displaystyle{n\over d}={\rm
deg}_yg$. The equation above  is called the $g$-adic expansion of
$f$.

\noindent This construction can be generalized to a sequence of
polynomials. Let to this end $n=d_1 > d_2 >...>d_h$ be a sequence
of integers such that $d_{i+1}$ divides $d_i$ for all $1\leq i\leq
h-1$, and set $e_i=\displaystyle{d_i\over d_{i+1}}$, $1\leq i\leq
h-1$ and $e_h=+\infty$. For all $1\leq i\leq h$, let $g_i$ be a monic polynomial of
$\RR[y]$ of degree $\displaystyle{n\over d_i}$ in $y$. Set
$G=(g_1,\ldots,g_h)$ and let $B=\lbrace
(\theta_1,\ldots,\theta_h)\in {\bf
  N}^{h}, 0\leq \theta_i < e_i$ for all $1\leq i\leq h\rbrace$. Then $f$
  can be
 uniquely written in the following form:

$$
 f=\sum_{\underline{\theta}\in
B}a_{\underline{\theta}}.g^{\underline{\theta}}
$$

\noindent where if
$\underline{\theta}=(\theta_1,\ldots,\theta_h)$, then
$g^{\underline{\theta}}=g_1^{\theta_1}.\ldots.g_h^{\theta_h}$ and
$a_{\underline{\theta}}\in\RR$. We call this expansion the
$G$-adic expansion of $f$. We set Supp$_G(f)=\lbrace
\underline{\theta}; a_{\underline{\theta}}\not=0\rbrace$ and we
call it the $G$-support of $f$.

\noindent  Let $f,g$ be as above and let
$f=g^d+\sum_{i=1}^{d}a_i.g^{d-i}$ be the $g$-adic expansion of
$f$. Assume that $d$ is a unit in ${\bf R}$. The Tschirnhausen
transform of $f$ with respect to $g$, denoted $\tau_f(g)$, is
defined by $\tau_f(g)=g+d^{-1}a_1$. Note that $\tau_f(g)=g$ if and
only if $a_1=0$. By [1],  $\tau_f(g)=g$ if and only if
deg$_y(f-g^d) <n-\dfrac{n}{d}$. If one of these equivalent
conditions is verified, then the polynomial $g$ is called a $d$-th
approximate root of $f$. By [1], there exists a unique
 $d$-th approximate root of $f$. We denote it by App$_y^d(f)$.

\section{The semigroup of a quasi-ordinary polynomial}

\medskip

\noindent Let ${\bf K}$ be an algebraically closed field of
characteristic zero, and let ${\bf R}={\bf K}[[x_1,\ldots,x_e]]$
(denoted ${\bf K}[[\underline{x}]]$) be the ring of formal power
series in $x_1,\ldots,x_e$ over ${\KK}$. Let
$f=y^n+a_1(\underline{x})y^{n-1}+\ldots+a_n(\underline{x})$ be
 a nonzero  polynomial of ${\bf R}[y]$. Suppose that the discriminant of $f$ is
 of the form $x_1^{N_1}.\ldots.x_e^{N_e}.u(x_1,\ldots,x_e)$, where
 $N_1,\ldots,N_e\in{\bf N}$ and $u(\underline{x})$ is a unit
 in ${\bf K}[[\underline{x}]]$. We call $f$  a
 quasi-ordinary polynomial. It follows from Abhyankar-Jung Theorem that there
 exists a  power series
 $y(\underline{t})=y(t_1,\ldots,t_e)\in {\bf K}[[t_1,\ldots,t_e]]$ (denoted ${\bf K}[[\underline{t}]]$)
 such that
$f(t_1^n,\ldots,t_e^n,y(\underline{t}))=0$. Furthermore, if $f$ is
an
 irreducible polynomial, then we have:

$$
f(t_1^n,\ldots,t_e^n,y)=\prod_{i=1}^n(y-y(w_1^it_1,\ldots,w_e^it_e))
$$

\noindent where  $(w_1^i,\ldots,w_e^i)_{1\leq i \leq n}$ are
distinct elements of $(U_n)^e$, $U_n$ being  the group of $n$th
roots of unity in ${\bf K}$.

\noindent  Suppose that $f$ is irreducible and let
$y(\underline{t})$ be as above. Write $y(\underline
t)=\sum_{p}c_p{\underline t}^p$ and define the support of $y$ to
be the set $\lbrace p|c_p\not= 0\rbrace$. Obviously the support of
$y(w_1t_1,\ldots,w_et_e)$ does not depend on $w_1,\ldots,w_e\in
U_n$. We denote it by Supp$(f)$ and we call it the support of $f$.
It is well known that there exists a finite sequence of elements
in Supp$(f)$, denoted $m_1,\ldots,m_h$, such that

i) $m_1 < m_2 <\ldots <m_h$, where $<$ means $<$ coordinate-wise.

ii) If $c_p\not= 0$, then $p\in (n{\bf Z})^e+\sum_{|m_i|\leq
|p|}m_i{\bf Z}$.

iii) $m_i\notin  (n{\bf Z})^e+\sum_{j<i}m_j{\bf Z}$ for all
$i=1,\ldots,h$.

\noindent The set of elements of this sequence is called the set
of characteristic exponents of $f$. We denote by convention
$m_{h+1}=(+\infty,\ldots,+\infty)$. If $e=1$, this set is nothing
but the set of Newton-Puiseux exponents of $f$.

\medskip

\noindent Let $u=\sum_{p}c_p{\underline t}^p$ in ${\bf
K}[[\underline t]]$ be a nonzero power series. We denote by
In$(u)$ the initial form of $u$: if $u=u_d+u_{d+1}+\ldots$ denotes
the decomposition of $u$ into  sum of homogeneous components, then
In$(u)=u_d$. We set $O_t(u)=d$ and we call it the
$\underline{t}$-order of $u$. We denote by exp$(u)$ the greatest
exponent of In$(u)$ with respect to the lexicographical order. We
denote by inco$(u)$ the coefficient $c_{\rm exp}(u)$, and we call
it the initial coefficient of $u$. We set M$(u)={\rm
inco}(u){\underline{t}}^{{\rm exp}(u)}$, and we call it the
initial monomial of $u$.

\noindent Let $g$ be a nonzero quasi-ordinary element of ${\bf
R}[y]$. The order of $g$ with respect to $f$, denoted $O(f,g)$, is
defined to be  exp($g(t_1^n,\ldots,t_e^n,y(\underline{t})$). Note
that it does not depend on the choice of the root
$y(\underline{t})$ of $f(t_1^n,\ldots,t_e^n,y)=0$. The set
$\lbrace O(f,g)|g\in {\bf R}\rbrace$ defines a subsemigroup of
${\bf Z}^e$. We call it the semigroup associated with $f$ and we
denote it by $\Gamma(f)$.

\medskip

\noindent Let $M(e,e)$ be the unit $(e,e)$ matrix. Let $D_1=n^e$
and for all $1\leq i\leq h$, let $D_{i+1}$ be the gcd of the
$(e,e)$ minors of the matrix $(nM(e,e),{m_1}^T,\ldots,{m_i}^T)$
(where $T$ denotes the transpose of a matrix). Since $m_i\notin
(n{\bf Z})^e+\sum_{j < i}m_j{\bf Z}$ for all $1\leq i\leq h $,
then $D_{i+1}< D_{i}$. We define the sequence $(e_i)_{1\leq i \leq
h}$ to be $\displaystyle{e_i={D_i\over D_{i+1}}}$ for all $1\leq
i\leq h$.

\noindent Let $M_0=({n\bf Z})^e$ and let $M_i= (n{\bf
Z})^e+\sum_{j=1}^im_j{\bf Z}$ for all $1\leq i\leq h$. Then $e_i$
is the index of the lattice $M_{i-1}$ in $M_i$, and
$n=e_1.\ldots.e_h$, in particular $D_{h+1}=n^{e-1}$. We set
$d_i=\displaystyle{{D_i}\over {D_{h+1}}}$ for all $1\leq i\leq
h+1$. In particular $d_1=n$ and $d_{h+1}=1$. The sequence
$(d_1,d_2,\ldots,d_{h+1})$ is called the gcd-sequence associated
with $f$. We also define the sequence $(r_k)_{1\leq k \leq  h}$ by
$r_1=m_1$ and $r_{k+1}=e_kr_k+m_{k+1}-m_k$ for all $1\leq k \leq
h-1$.


\medskip

\noindent Denote by Root$(f)$ the set of $n$ roots of
$f(t_1^n,\ldots,t_e^n,y)=0$ introduced  above and let
$y(\underline{t})$ be an element of this set. We have the
following:

{\begin{lemma}{\rm i) In$(y(\underline{t})-z(\underline{t}))$ is a
monomial for all $z(\underline{t})\in {\rm Root}(f)-\lbrace
y(\underline{t})\rbrace$. Furthermore, $\lbrace {\rm
exp}(y(\underline{t})-z(\underline{t}))|z(\underline{t})\in {\rm
Root}(f)-\lbrace y(\underline{t})\rbrace \rbrace = \lbrace
m_1,\ldots,m_h\rbrace$.

ii) Let for all $1\leq k\leq h$,

 $$S(k)=\lbrace z(\underline{t})\in {\rm Root}(f)| {\rm exp}(y(\underline{t})-z(\underline{t})) = m_k \rbrace. $$

 $$R(k)=\lbrace  z(\underline{t})\in {\rm Root}(f)| {\rm exp}(y(\underline{t})-z(\underline{t}))\geq m_k \rbrace.$$

$$Q(k)= \lbrace z(\underline{t})\in {\rm Root}(f)| {\rm exp}(y(\underline{t})-z(\underline{t}))< m_k \rbrace.$$

\noindent Then the cardinality of $S(k)$ (resp. $R(k)$, resp.
$Q(k)$) is $d_k-d_{k+1}$. (resp. $d_k$, resp. $n-d_k$).
}\end{lemma}

\begin{proof}{.}{\rm The proof is the same as in the case of plane curves. Note that
given $z(\underline{t})\in {\rm Root}(f)$, since
$y(\underline{t})-z(\underline{t})$ divides the discriminant, then
$y(\underline{t})-z(\underline{t})=a.{\underline{t}}^{m}.u$, where
$a\in {\bf K}^*, m$ is a characteristic exponent of $f$, and $u$
is a unit in ${\bf K}[[\underline t]]$. In particular,
In$(y(\underline{t})-z(\underline{t}))=a.{\underline{t}}^m$.}
\end{proof}

 \noindent Let
 $\phi(\underline t)=(t_1^p,\ldots,t_e^p,Y(\underline{t}))$ and
 $\psi(\underline t)=(t_1^q,\ldots,t_e^q,Z(\underline{t}))$
be two nonzero elements of ${\bf K}[[\underline t]]^{e+1}$. We
define the contact between $\phi$ and $\psi$ to be the element
${\dfrac{1}{pq}}{\rm
exp}(Y(t_1^q,\ldots,t_e^q)-Z(t_1^p,\ldots,t_e^p))$. We denote it
by c$(\phi,\psi)$.

\noindent We define the contact between $f$ and $\phi$, denoted
c$(f,\phi)$, to be the maximal element in the set of contacts
between $\phi$ with the roots of $f(t_1^n,\ldots,t_e^n,y)=0$.

\noindent Let
$g=y^m+b_1(\underline{x})y^{m-1}+\ldots+b_m(\underline{x})$ be a
nonzero polynomial of ${\bf R}[y]$. Suppose that $g$ is an
irreducible quasi-ordinary polynomial and let $\psi(\underline
t)=(t_1^m,\ldots,t_e^m,Z(\underline{t}))$ be a root of
$g(x_1,\ldots,x_e,y)=0$. We define the contact between $f$ and
$g$, denoted c$(f,g)$, to be the contact between $f$ and $\psi$,
 and we recall that this definition
does not depend on the choice of the root $\psi$ of $g$. Note that
if $f.g$ is a quasi-ordinary polynomial, then ${\rm
In}(f(\psi(\underline{t}))=M(f(\psi(\underline{t}))$.

\noindent With these notations we have the following proposition:

\begin{proposition}{\rm  Let $g=y^m+b_1(\underline{x})y^{m-1}+\ldots+b_m(\underline{x})$ be an
 irreducible quasi-ordinary polynomial of ${\bf
R}[y]$ and suppose that $f.g$ is a quasi-ordinary polynomial. Let
$(D'_j)_{1\leq j\leq h'+1}$ (resp. $(d'_j)_{1\leq j\leq h'+1}$,
$(m'_j)_{1\leq j\leq h'}$)  be the set of characteristic sequences
associated with $g$. If $c$ denotes the contact c$(f,g)$, then we
have the following:

i) If for all $1\leq q\leq h, nc\notin M_q$, then $O(f,g)=n.m.c$.

ii) Otherwsie, let $1\leq q\leq h$ be the smallest integer
such that $nc\in M_{q}$, then
$O(f,g)=(r_qd_q+(nc-m_q)d_{q+1}).{\dfrac{m}{n}}$.

iii) If $nc\in M_{q}-M_{q-1}$ and $nc\not= m_q$,  then
$\dfrac{n}{d_{q+1}} |m$.}\end{proposition}

\begin{proof}{.}{\rm i) and ii) are obvious. To prove iii) let $\phi=(t_1^n,\ldots,t_e^m,Y(\underline{t}))$
(resp. $\psi=(t_1^m,\ldots,t_e^m,Z(\underline{t}))$) be a root of
$f(\underline{x},y)=0$ (resp. $g(\underline{x},y)=0$) and  remark
that if $nc\in M_{q}-M_{q-1}$ and $nc\not= m_q$ then the exponents of
$Z({t_1}^n,...,{t_e}^n)$ coincide with those of
$Y({t_1}^m,...,{t_e}^m)$ till at least $m_q.m$. Write
$Y(\underline{t})=\sum_ic_i{\underline{t}}^i$ and
$Z(\underline{t})=\sum_jc'_j\underline{t}^j$, then for all $i\in
M_{q+1}$ in Supp$(Y)$, there exists $j\in {\rm Supp}(Z)$ such that
$i.m=j.n$. But the gcd of minors of the matrix
$(m.nM(e,e),t_{m.m_1},\ldots,t_{m.m_q})$ is $m^e.D_{q+1}$, and the
gcd of minors of the matrix
$(m.nM(e,e),t_{n.m'_1},\ldots,t_{n.m'_q})$ is
 $n^e.D'_{q+1}$. Thus
$m^e.D_{q+1}=n^e.D'_{q+1}$, in particular
$m^e.n^{e-1}d_{q+1}=n^e.m^{e-1}.d'_{q+1}$. This implies that
$m=\displaystyle{n\over d_{q+1}}.d'_{q+1}$, which proves our
assertion.}
\end{proof}

\section{Pseudo roots and generators of the semigroup}

 \noindent Let the notations be as in section 2 and let $q \in \NN,1 \leq q \leq h+1$.
 Let
 $y(\underline{t})=\sum c_p\underline{t}^p\in {\rm Root}(f)$ and
consider the truncation $\bar{y}(\underline{t})=\sum_{p\in
M_q}{c_p.\underline{t}^p}$ of $y$. Let $G_q(\underline{x},y) \in
\RR[y]$ be the minimal polynomial of
$\bar{y}(\underline{x}^{1\over n})$ over $\KK((\underline{x}))$.
Then $G_q$ is a quasi-ordinary polynomial of degree
$\dfrac{n}{d_q}$ in $y$, and $G_q(t_1^{n\over
d_q},\ldots,t_e^{n\over d_q},\bar{y}(\underline{t}^{1\over
d_q}))=0$. Furthermore, there exist $\displaystyle{n\over d_q}$
distinct elements $(\rho_1^i,\ldots,\rho_{n\over d_q}^i)_{1\leq
i\leq {n\over d_q}}$ in $(U_{n\over d_q})^e$,  where  $U_{n\over
d_q}$ denotes the set of $\displaystyle{n\over d_q}$th roots of
unity in ${\bf K}$, such that:

$$
G(t_1^{n\over d_q},\ldots,t_e^{n\over d_q},y)= \prod_{i=1}^{n\over
d_q}(y-\bar{y}(\rho_1^it_1^{1\over d_q},\ldots,\rho_{n\over
d_q}^it_e^{1\over d_q}))
$$

\noindent We call $G_q$ a $d_q$th pseudo root of $f$. With the
notations of Section 2, $c(f,G_q)=m_q$, and consequently by
Proposition 2.2. ii), $O(f,G_q)=r_q$.

\noindent Let $G=(G_1,\ldots,G_h,G_{h+1})$ be a set of $d_k$th
pseudo roots of $f$, $1\leq k\leq h+1$, and recall that
deg$_y(G_1)=1$ and that $G_{h+1}=f$. Let $B(G)=\lbrace
\underline{\theta}\in {\bf N}^{h+1}; 0\leq \theta_k < e_k$ for all
$1\leq k\leq h$ and $\theta_{h+1}< +\infty\rbrace$. Given two
elements $\underline{\theta}^1,\underline{\theta}^2 \in B(G)$, and
two elements $\underline{\gamma}^1,\underline{\gamma}^2\in {\bf
N}^e$, if $\theta_{h+1}^1=\theta_{h+1}^2$ and
$\underline{\theta}^1\not=\underline{\theta}^2$ then
$\sum_{i=1}^e\gamma^1_i.r_0^i+\sum_{k=1}^h\theta^1_kr_k\not=
\sum_{i=1}^e\gamma^2_i.r_0^i+\sum_{k=1}^h\theta^2_kr_k$.

\noindent Let $F(\underline{x},y)$ be a monic polynomial of ${\bf
R}[y]$ and let:

$$
F=\sum_{\underline{\theta}\in
B(G)}c_{\theta}(\underline{x})G_1^{\theta_1}.\ldots.G_h^{\theta_h}.G_{h+1}^{\theta_{h+1}}
$$

\noindent be the $G$-adic expansion of $F$.  Let
Supp$_G(F)=\lbrace \underline{\theta}\in B(G), c_{\theta}\not=
0\rbrace$ and let $B'(G)=\lbrace \underline{\theta}\in {\rm
Supp}_G(F); \theta_{h+1}=0\rbrace$. Clearly $f$ divides $F$ if and
only if $B'(G)=\emptyset$. Otherwise, there is a unique
$\underline{\theta}^0\in {\rm Supp}_G(F)$ such that
$O(f,F)=O(f,c_{\underline{\theta}}(\underline{x})G_1^{\theta_1}.\ldots.G_h^{\theta_h})=O(f,c_{\underline{\theta}}(\underline{x}))+\sum_{i=1}^h\theta_ir_i$.
In particular, $r_0^1,\ldots,r_0^e, r_1,\ldots,r_h$ generate
$\Gamma(f)$.





\section{Approximate roots of a quasi-ordinary polynomial}

\noindent Let the notations be as in Section 2, and let
$y(\underline{t})=\sum_{p}c_{p}{\underline{t}}^p\in$ Root$(f)$.
Given $1\leq q\leq h$ and $z(\underline{t})\in {\rm Root}(f)$,
there exists $w(z)\in U_n$ such that the coefficient of $t^{m_q}$
in the expansion of $z(\underline{t})$ is $w(z).c_{m_q}$. Let
$Q(q)$ (resp. $R(q)$, resp. $S(q)$) be the set of elements of
Root$(f)$ whose contact with $y(\underline{t})$ is $< m_q$ (resp.
$\geq m_q$, resp. $=m_q$) and let $\zeta$ be an element of $\KK$.
It follows from Lemma 2.1. that:

 $$
 \prod_{z(\underline{t})\in
 R(q)}{(\zeta-w(z).c_{m_q})}=(\zeta^{e_q}-c_{m_q}^{e_q})^{d_{q+1}}
 $$

\noindent On the other hand, if $q\geq 2$, since:

$$
\prod_{z(\underline{t}) \in
Q(q)}{(y(\underline{t})-z(\underline{t}))}
=\prod_{k=1}^{q-1}\prod_{z(\underline{t}) \in
S(k)}{(y(\underline{t})-z(\underline{t}))}
$$

\noindent then

$$
{\rm exp}(\prod_{z(\underline{t}) \in
Q(q)}{(y(\underline{t})-z(\underline{t}))})=\sum_{k=1}^{q-1} {\rm
exp}(\prod_{z(\underline{t}) \in
S(k)}{(y(\underline{t})-z(\underline{t})}))
$$
$$
=\sum_{k=1}^{q-1}(d_k-d_{k+1}).m_k
=m_1d_1+\sum_{k=1}^{q-2}(m_{k+1}-m_k)d_{k+1}-m_{q-1}d_q
$$

$$
=r_1d_1+\sum_{k=1}^{q-2}(r_{k+1}d_{k+1}-r_kd_k)-m_{q-1}d_q=r_{q-1}.d_{q-1}-m_{q-1}.d_q
$$

\noindent Consequently

\[{\rm exp}(\prod_{z(\underline{t}) \in
Q(q)}{(y(\underline{t})-z(\underline{t}))})=
\begin {cases}r_{q-1}.d_{q-1}-m_{q-1}.d_q & \text{ if $q \geq 2$}\\
0 &\text{if $q=1$}
\end {cases}
\]




\noindent Let $Z$ be an indeterminate and define a $(q,Z)$
deformation of $y(\underline{t})$  to by any $y^*(Z,\underline{t})
\in \KK(Z)[[\underline{t}]]$  such that
$$
{\rm In}(y^*(Z,\underline{t})-\sum_{p\in
M_q}{c_p.\underline{t}^p})=Z.\underline{t}^{m_q}.
$$

\noindent Equivalently a $(q,Z)$-deformation
$y^*(Z,\underline{t})$ of $y(\underline{t})$ is any element
$y^*(Z,\underline{t})\in \KK(Z)[[\underline{t}]]$ such that:

$$
y^*(Z,\underline{t})=y(\underline{t})+(Z-c_{m_q}).\underline{t}^{m_q}+u(Z,\underline{t})
$$

\noindent where $O_{\underline{t}}(u(Z,\underline{t})) >|m_q|$.
Let $z(\underline{t})\in {\rm Root}(f)$  and let
$y^*(Z,\underline{t})$ be a $(q,Z)$ deformation of
$y(\underline{t})$. We want to calculate the contact between
$y^*(Z,\underline{t})$ and $z(\underline{t})$.  Note that:

$$
y^*(Z,\underline{t})-z(\underline{t})=(Z-c_{m_q}).\underline{t}^{m_q}+y(\underline{t})-z(\underline{t})+u(\underline{t},Z)
$$

\noindent It follows that if  $z(\underline{t})\in Q(q)$, then
${\rm In} (y^*(Z, \underline{t})-z(\underline{t}))={\rm
In}(y(\underline{t})-z(\underline{t}))$. In particular:


 \[(1) \quad {\rm In}\prod_{z(\underline{t})\in
Q(q)}{(y^*(Z,\underline{t})-z(\underline{t}))} =
\begin {cases}  a_1  &\text {if $q=1$}\\
a_qt^{r_{q-1}.d_{q-1}-m_{q-1}d_q} & \text  {if $q \geq 2$}
\end {cases}
\]

\noindent where for all $q\geq 1$, $a_q$ is a nonzero constant of
$\KK$. On the other hand, if  $z(\underline{t}) \in R(q)$, then
exp$(z(\underline{t})-y(\underline{t})) \geq m_q$, then
Inco$(y^*(Z, \underline{t})-z(\underline{t}))=
 (Z-c_{m_q})+(c_{m_q}-w(z)c_{m_q})=Z-w(z)c_{m_q}$, in particular
 ${\rm In}(y^*(Z, \underline{t})-z(\underline{t}))= (Z-w(z)c_{m_q})\underline{t}^{m_q}$. Consequently

$$
{\rm In}(\prod_{z(\underline{t})\in
R(q)}{(y^*(Z,\underline{t})-z(\underline{t}))})
=(Z^{e_q}-c_{m_q}^{e_q})^{d_{q+1}}.\underline{t}^{m_q.d_q}.
$$
\noindent Now

$$
f(t_1^n,\ldots,t_e^n,y^*(Z,\underline{t}))=\prod_{z(\underline{t})\in
Q(q)}{(y^*(Z,\underline{t})-z(\underline{t}))}.\prod_{z(\underline{t})\in
R(q)}{(y^*(Z,\underline{t})-z(\underline{t}))}
$$

\noindent and $r_qd_q=r_{q-1}d_{q-1}+m_qd_q-m_{q-1}d_q$, in
particular:

$$
{\rm In}(f(t_1^n,\ldots,t_e^n,y^*))=\alpha
(Z^{e_q}-y_{m_q}^{e_q})^{d_{q+1}}.\underline{t}^{r_q.d_q}.
$$

\noindent where $\alpha\in {\bf K}^*$.

\begin {lemma}{\rm Let $q \in \NN, 1 \leq q \leq h$.
Let $F=F(\underline {x},y) \in \RR [y]$  such that deg$_yF
<\dfrac{n}{d_q}$. Let $y^*(Z,\underline{t})$ be a
$(q,Z)$-deformation of $y(\underline{t})$. Then
inco$(F(t_1^n,\ldots,t_e^n,y^*(Z,\underline{t})))\in{\bf K}^*$.}

\end {lemma}

\begin {proof}{.} If  $q=1$,  then deg$_y(F)=0$, in particular $F(\underline{x},y) \in \RR$, and
 $F(t_1^n,\ldots,t_e^n,y^*(Z,\underline{t})) \in \KK[[\underline{t}]])$. Let  $q \geq 2$  and for all $1\leq k < q$, let
$G_k(\underline{x},y)$ be a pseudo $d_k$th root of $f$. Let
$G^q=(G_1,\ldots,G_{q-1})$ and let $B(G^q)=\lbrace
(\theta_1,\ldots,\theta_{q-1}); 0 \leq \theta_k < e_k$ for all
$1\leq k < q\rbrace$. Let:

$$
F=\sum_{\underline{\theta}\in
B(G^q)}{c_{\underline{\theta}}(\underline{x}).G_1^{\theta_1}.\ldots.G_{q-1}^{\theta_{q-1}}}
$$

\noindent be the $G^q$-adic expansion of $F$. Since $O(f,G_k)={\rm
exp}(G_k(t_1^n,\ldots,t_e^n,y(\underline{t})))=r_k$ and
$c(f,G_k)=m_k < m_q$, then
exp$(G_k(t_1^n,\ldots,t_e^n,y^*(Z,\underline{t})))=r_k$. In
particular there is a unique ${\underline{\theta}}^0\in B(G^q)$
such that:

$$
{\rm exp}(F(t_1^n,\ldots,t_e^n,y^*(Z,\underline{t})))={\rm
exp}(c_{\underline{\theta}^0}(t_1^n,\ldots,t_e^n))+\sum
_{k=1}^{q-1}{\theta^0_k.{\rm
exp}(G_k(t_1^n,\ldots,t_e^n,y^*(Z,\underline{t})))} $$
$$={\rm
exp}(c_{\underline{\theta}_0}(t_1^n,\ldots,t_e^n))+\sum
_{k=1}^{q-1}{\theta^0_kr_k}.
$$

\noindent   In particular:

$$
{\rm In} (F(t_1^n,\ldots,t_e^n,y^*(Z,\underline{t})))={\rm
In}(c_{\underline{\theta}^0}(t_1^n,\ldots,t_e^n).(G_1^{\theta_1}.\ldots.G_{q-1}^{\theta_{q-1}})(t_1^n,\ldots,t_e^n,y^*(Z,\underline{t}))).
$$

\noindent But ${\rm
inco}((c_{\underline{\theta}^0}(t_1^n,\ldots,t_e^n))\in {\bf K}^*$
and by (1), inco$(g_k(t_1^n,\ldots,t_e^n,y^*(Z,\underline{t}))\in
{\bf K}^*$ for all $1\leq k\leq q-1$. This implies our assertion.

\end {proof}

\begin {lemma}{\rm Let $q\in \NN, 2\leq q\leq h$ and let $g=g(\underline{x},y) \in
\RR[y]$  be a  monic polynomial of degree $\dfrac{n}{d_q}$ in $y$.
Let $y^*(Z, \underline{t})$ be a $(q,Z)$-deformation of $y(t)$. If
${\rm In}(g(t_1^n,\ldots,t_e^n,y^*(Z,\underline{t})))=
\alpha.Zt^{r_q}$,  $\alpha\in {\bf K}^*$,  then ${\rm In}(
(\tau_fg)(t_1^n,\ldots,t_e^n,y^*(Z,\underline{t})))=\alpha.Zt^{r_q}$.}

\end {lemma}

\begin {proof}{.} Let

$$f
=g^{d_q}+a_1g^{d_q-1}+\ldots+a_{d_q}
$$

\noindent  be the $g$-adic expansion of $f$, and  recall that
$\tau_f(g)=g+d_q^{-1}a_1$. We need to show that $r_q < {\rm
exp}(a_1(t_1^n,\ldots,t_e^n,y^*(Z,\underline{t})) )$. We
 have

 $$
 f(t_1^n,\ldots,t_e^n,y^*(Z,\underline{t}))=\sum_{k=0}^{d_q}{a_k(t_1^n,\ldots,t_e^n,y^*(Z,\underline{t})).g^{d_q-k}(t_1^n,\ldots,t_e^n,y^*(Z,\underline{t}))}
 $$

 \noindent where $a_0=1$. Let

 $$
 u= {\rm inf }\lbrace {\rm exp}(a_k(t_1^n,\ldots,t_e^n,y^*(Z,\underline{t})).g^{d_q-k}(t_1^n,\ldots,t_e^n,y^*(Z,\underline{t})));
  0 \leq k\leq d_q \rbrace.
 $$

\noindent Since $a_0=1$ and exp
$(g^{d_q}(t_1^n,\ldots,t_e^n,y^*(Z,\underline{t})))=d_q.r_q$, then
$u\in{\bf N}^e$. Let

$$
I= \lbrace 0 \leq k \leq d_q; {\rm exp}
(a_k(t_1^n,\ldots,t_e^n,y^*(\underline{t},Z)).g^{d_q-k}(t_1^n,\ldots,t_e^n,y^*(\underline{t},Z)))=u
\rbrace.
$$

\noindent then for all $k \in I$,
$a_k(t_1^n,\ldots,t_e^n,y^*(Z,\underline{t}))\not =0$ and, by
lemma 4.1., inco$(a_k(t_1^n,\ldots,t_e^n,y^*(Z,\underline{t}))) =
\alpha_k \in \KK^*$. Consequently inco$(
a_k(t_1^n,\ldots,t_e^n,y^*(Z,\underline{t})).g^{d_q-k}(t_1^n,\ldots,t_e^n,y^*(Z,\underline{t})))=\alpha_k.\alpha^{d_q-k}.Z^{d_q-k}$
for all $k\in I$. In particular In$(f(t_1^n,
\ldots,t_e^n,y^*(Z,\underline{t})))=(\sum_{k \in
I}{\alpha_k.\alpha^{d_q-k}.Z^{d_q-k}}).t^u$. But

$$
{\rm In}(f(t_1^n, \ldots,t_e^n,y^*(\underline{t},Z)))=a
(Z^{e_q}-y_{m_q}^{e_q})^{d_{q+1}}.t^{r_q.d_q}, a\in\KK^*
$$

\noindent so:

$$
u=r_q.d_q \quad {\rm and} \quad \sum_{k \in
I}\alpha_k.\alpha^{d_q-k}.Z^{d_q-k}=a(Z^{e_q}-y_{m_q}^{e_q})^{d_{q+1}},
$$

\noindent in particular $\sum_{k \in
I}\alpha_k.\alpha^{d_q-k}.Z^{d_q-k}\in \KK [Z^{e_q}]$. On the
other hand $e_q=\dfrac{d_q}{d_{q+1}}$ doesn't divide $d_q-1$, then
$d_q-1 \not \in I$, so  $u <{\rm exp}
(a_1(t_1^n,\ldots,t_e^n,y^*(Z,\underline{t})).g^{d_q-1}(t_1^n,\ldots,t_e^n,y^*(Z,\underline{t})))={\rm
exp}(a_1(t_1^n,\ldots,t_e^n,y^*(Z,\underline{t})))+(d_q-1).r_q$.
This  proves our assertion.

\end {proof}

\noindent As a corollary we get the following theorem:

\begin{theorem}{\rm Let the notations be as above, and let $d_1,\ldots,d_{h}, d_{h+1}=1$ be the
gcd-sequence of $f$. Then $O(f,{\rm App}_{d_k}(f))=r_k$ for all
$1\leq k\leq h$.}
\end{theorem}

\begin{proof}{.} For all $1\leq k\leq h$, let $G_k$ be a pseudo
$d_k$th root of $f$. Then deg$_y(G_k)=\displaystyle{n\over d_k}$.
But App$_{d_k}(f)=\tau_f(G_k)$. Now use Lemma 4.2.
\end{proof}

\section{Generalized Newton polygons}

\medskip

\noindent Let $n\in {\bf N}$ and let
${\underline{r}}_0=(r_0^1,\ldots,r_0^e)$ be the canonical basis of
$(n{\bf Z})^e$. Let $r_1<\ldots <r_h$ be a sequence of elements of
${\bf N}^e$, where $<$ means $<$ coordinate-wise. Set $D_1=n^e$
and for all $1\leq k\leq h$, let $D_{k+1}$ be the GCD of the
$(e,e)$ minors of the $(e,e+k)$ matrix
$(n.I(e,e),(r_1)^T,\ldots,(r_k)^T)$. Suppose that $n^{e-1}$
divides $D_{k}$ for all $1\leq k\leq h+1$ and that
$D_{h+1}=n^{e-1}$, and also that $D_1
> D_1 >\ldots > D_{h+1}$, in such a way that if we set $d_1=n$ and
$\displaystyle{d_k={D_k\over n^{e-1}}}$ for all $2\leq k \leq h$,
then $d_1=n > d_2 >\ldots > d_{h+1}=1$.

\medskip

\noindent  For all $1\leq k\leq h$, let $g_k$ be a monic
polynomial of degree $\displaystyle{n\over d_k}$ in $y$ and set
$G=(g_1,\ldots,g_h)$. Let $F$ be a nonzero polynomial of ${\bf
K}[[\underline{x}]][y]$ and let:

$$
F=\sum_{\underline{\theta}\in
B(G)}c_{\underline{\theta}}(\underline{x})g_1^{\theta_1}.\ldots.g_h^{\theta_h}
$$
\noindent where $B(G)=\lbrace \underline{\theta}=
(\theta_1,\ldots,\theta_h); \forall 1\leq i\leq h-1, 0\leq
\theta_i < e_i=\dfrac{d_i}{d_{i+1}}$ and $\theta_{h} <
+\infty\rbrace$, be the $G$-adic expansion of $F$. Let
Supp$_G(F)=\lbrace \underline{\theta}\in B(G);
c_{\underline{\theta}}\not= 0\rbrace$. If $\theta\in {\rm
Supp}_G(F)$ and $\underline{\gamma}={\rm
exp}(c_{\underline{\theta}}(\underline{x}))$, we shall associate
with the monomial
$c_{\underline{\theta}}(\underline{x})g_1^{\theta_1}.\ldots.g_{h}^{\theta_{h}}$
the $e$-uplet

$$
<((\underline{\gamma},\underline{\theta}),({\underline{r}}_0,\underline{r}))>=\sum_{i=1}^e\gamma_i.r_0^i+\sum_{j=1}^{h}\theta_j.r_j
$$

\noindent There is a unique $\underline{\theta}^0\in {\rm
Supp}_G(F)$ such that if ${\underline{\gamma}}^0={\rm
exp}(c_{{\underline{\theta}}^0}(\underline{x}))$, then:

$$
<(({\underline{\gamma}}^0,{\underline{\theta}}^0),({\underline{r}}_0,\underline{r}))>=
{\rm inf}\lbrace
<((\gamma,\underline{\theta}),(r_0,\underline{r}))>,
\underline{\theta}\in {\rm Supp}_G(F)\rbrace
$$

\noindent  We set

$$
{\rm
fO}(\underline{r},G,F)=<(({\underline{\gamma}}^0,{\underline{\theta}}^0),({\underline{r}}_0,\underline{r}))>
$$

\noindent and we call it the formal order of $F$ with respect to
$(\underline{r},G)$. We also set:

$$
M_G(F)=M(c_{{\underline{\theta}}_0}).g_1^{\theta_1^0}.\ldots.g_h^{\theta_h^0}
$$

\noindent and we call it the initial monomial of $F$ with respect
to $(\underline{r},G)$.

\medskip

 \noindent Let
$f=y^n+a_1(\underline{x})y^{n-1}+\ldots+a_n(\underline{x})$ be a
quasi-ordinary polynomial of ${\bf K}[[x_1,\ldots,x_e]][y]$ and
let $d\in {\bf N}$ be a divisor of $n$. Let $g$ be a monic
polynomial of ${\bf K}[[x_1,\ldots,x_e]][y]$ of degree
$\displaystyle{{n\over d}}$ in $y$ and let:

$$
f=g^{d}+a_1(\underline{x},y)g^{d-1}+\ldots+a_{d}(\underline{x},y)
$$

\noindent be the $g$-adic expansion of $f$. We associate with $f$
the set of points:

$$
\lbrace ({\rm fO}(\underline{r},G,a_k),(d-k){\rm
fO}(\underline{r},G,g)), k=0,\ldots,d\rbrace \subseteq {\bf
N}^{e}\times {\bf N}^e
$$

\noindent We denote this set by GNP$(f,\underline{r},G,g)$ and we
call it the generalized Newton polygon of $f$ with respect to
$(\underline{r},G,g)$. Note that if $e=1$ and $f$ is an
irreducible polynomial of ${\bf K}[[x]][y]$, then the above set is
equivalent to the usual Newton polygon of $f$.

\begin{definition}{\rm We say that $f$ is straight with respect to $(\underline{r},G,g)$ if the following holds:

i) fO$((\underline{r},G,a_d)=d.{\rm fO}((\underline{r},G,g))$.

ii) For all $1\leq k\leq h-1$, fO$(\underline{r},G,a_k)\geq k.{\rm
fO}((\underline{r},G,g))$, where $\geq$ mean $\geq$
coordinate-wise.

\noindent We say that $f$ is strictly straight with respect to
$(\underline{r},G,g)$ if the inequality in ii) is a strict
inequality.}
\end{definition}

 \section{\bf The criterion}

 \medskip

 \noindent Let $f=y^n+a_1(x)y^{n-1}+\ldots+a_n(x)$  be a nonzero element of
 ${\bf K}[[x_1,\ldots,x_e]][y]$ and assume, after an eventual change of
variables, that $a_1(\underline{x})=0$.  Let
${\underline{r}}_0=(r^1_0,\ldots,r^e_0)$ be the canonical basis of
$(n{\bf Z})^e$ and let $d_1=n$. Let $g_1=y$ be the $d_1$-th
approximate root of $f$ and set $m_1=r_1={\rm
exp}(a_n(\underline{x}))$. Let $D_2$ be the gcd of the $(e,e)$
minors of the $(e,e+1)$ matrix $(n.I(e,e),{m_1}^T)$. Let
$d_2=\displaystyle{D_2\over n^{e-1}}$ and let $g_2$ be the
$d_2$-th approximate root of $f$ and set
$e_2=\displaystyle{{d_1\over d_2}={n\over d_2}}$.... Suppose that
we constructed $(r_1,\ldots,r_{k-1})$, $(m_1,\ldots,m_{k-1})$, and
$(d_1,\ldots,d_{k})$, then let $g_k$ be the $d_k$-th approximate
root of $f$ and let

$$
f=g_k^{d_k}+\beta_2^kg_k^{d_k-2}+\ldots+\beta_{d_k}^k
$$

\noindent be the $g_k$-adic expansion of $f$. Then $r_k={\rm
fO}(\underline{r}^k,G^k,\beta_{d_k}^k)$, where
$\underline{r}^k=(\dfrac{r^1_0}{d_k},\ldots,\dfrac{r^e_0}{d_k},
\dfrac{r_1}{d_k},\ldots,\dfrac{r_{k-1}}{d_k})$ and
$G^k=(g_1,\ldots,g_{k-1})$. With these notations we have the
following:

\begin{theorem}{\rm The polynomial $f$ is an irreducible quasi-ordinary polynomial
if and only if the following holds:

i) There is an integer $h$ such that $d_{h+1}=1$.

ii) For all $1\leq k\leq h-1, r_kd_k < r_{k+1}d_{k+1}$, where $<$
means $<$ coordinate-wise.

iii) For all $2\leq k\leq h+1$, $g_k$ is strictly straight with
respect to $(\underline{r}^k,G^k,g_{k-1})$.}
\end{theorem}

\noindent We shall first prove the following results:

\begin{lemma}{\rm  Let $c\in {\bf K}^*$. The quasi-ordinary polynomial
$F=y^n-cx_1^{\alpha_1}.\ldots.x_e^{\alpha_e}$ is irreducible in
${\bf K}[[x_1,\ldots,x_e]][y]$ if and only if
gcd$(n,\alpha_1,\ldots,\alpha_e)=1$, or equivalently if and only
if the gcd of the $(e,e)$ minors of the matrix
$(nI(e,e),(\alpha_1,\ldots,\alpha_e)^T)$ is $n^{e-1}$.}
\end{lemma}

\begin{proof}{.} Let $\tilde{c}$ be an n-th root of $c$ in ${\bf K}$ and let
$Y=\tilde{c}x_1^{\alpha_1\over n}.\ldots.x_e^{\alpha_e\over n}\in
{\bf K}((x_1^{1\over n},\ldots,x_e^{1\over n}))$. Then $F$ is the
minimal polynomial of $Y$ over ${\bf K}((x_1,\ldots,x_e))$. In
particular it is irreducible.
\end{proof}

\begin{proposition}{\rm Assume that the polynomial $f$ is
irreducible and let $(m_k)_{1\leq k\leq h}$ be the set of
characteristic exponents of $f$. Let $F$ be a quasi-ordinary
polynomial of  ${\bf K}[[x_1,\ldots,x_e]][y]$ and assume that $F$
is monic of degree $n$ in $y$. If $O(f,F) > r_hd_h$, then $F$ is
irreducible in ${\bf K}[[x_1,\ldots,x_e]][y]$. }
\end{proposition}

\begin{proof}{.} Assume that $F$ is not irreducible and let  $\tilde{F}$ be an irreducible component of
$F$ in ${\bf K}[[x_1,\ldots,x_e]][y]$. Let $C=c(f,\tilde{F})$ be
the contact of $f$ with $\tilde{F}$. If $C\in M_{h+1}$ and
$C\not=m_h$, then deg$_y(\tilde{F}) \geq n$, which is a
contradiction because $F$ is not irreducible. In particular,
$O(f,\tilde{F})\leq r_hd_h.\displaystyle{{{\rm
deg}_y(\tilde{F})\over n}}$. Since this is true for all
irreducible component of $F$, then $O(f,F)\leq
r_hd_h.\displaystyle{{{\rm deg}_y(F)\over n}}=r_hd_h$, which is a
contradiction.
\end{proof}

\begin{proofs}{.} Suppose first that $f$ is irreducible. Then the  condition i) is
obvious. On the other hand, if we denote by $(m_k)_{1\leq k\leq
h}$ the set of characteristic exponents of $f$, then

$$
r_{k+1}d_{k+1}=r_kd_k+(m_{k+1}-m_k).d_{k+1}
$$

\noindent for all $1\leq k\leq h-1$. This proves ii). Now for all
$1\leq k\leq h+1$, $g_k$ is an irreducible quasi-ordinary
polynomial  and $g_1,\ldots,g_{k-1}$ are the approximate roots of
$g_k$. In particular, to prove iii), it suffices to prove that
$f=g_{h+1}$ is straight with respect to
$(\underline{r},G,g_h)=(\underline{r}^{h+1},G^{h+1},g_h)$. Let

$$
f=g_h^{d_h}+\beta_2^hg_h^{d_h-2}+\ldots+\beta_{d_h}^h
$$

\noindent be the $g_h$-adic expansion of $f$. If we denote by
$\Gamma^h$ the semigroup generated by
$r^0_1,\ldots,r^0_e,r_1,\ldots,r_{h-1}$, then we have the
following:

- For all  $2\leq i\leq h-1$, $O(\beta_i^h,f)\in \Gamma^h$.

- For all $0 < a < d_h, a.r_h\notin \Gamma^h$.

\noindent It follows that for all  $2\leq i\leq h-1,
O(\beta_i^h,f)\not= i.r_h$  and for all  $2\leq i\not= j\leq
d_h-1$, $O(\beta_i^h,f)+(d_h-i)r_h\not=O(\beta_j^h,f)+(d_h-j)r_h$.
Since $O(g_h^{d_h},f)=r_hd_h$, then $O(\beta_{d_h}^h,f)=r_hd_h$
and $O(\beta_i^h,f) > i.r_h$ for all $2\leq i\leq d_h-1$. This
implies iii).

\medskip

\noindent Conversely suppose that $f$ verifies the conditions i),
ii), and iii).  We shall prove by induction on $h$ that $f$ is
irreducible. Suppose first that $h=1$, then
$f=y^n+a_2(\underline{x})y^{n-2}+\ldots+a_n(\underline{x})$ and
$O_x(a_i(\underline{x}))
> i. O_x(a_n(\underline{x}))$ for all $2\leq i\leq n-1$. Furthermore, $D_2=n^{e-1}$. In particular
$F=y^n+M(a_1(\underline{x}))$ is irreducible by Lemma 6.2. But
$O(F,f)=O(f-F,f)> r_1d_1$, then $f$ is irreducible by Proposition
6.3.

\noindent Let $h >1$ and assume that $g_k$ is an irreducible
quasi-ordinary polynomial for all $1\leq k\leq h$. Let

$$
f=g_h^{d_h}+\beta_2^hg_h^{d_h-2}+\ldots+\beta_{d_h}^h
$$

\noindent be the $g_h$-adic expansion of $f$ and let
$F=g_h^{d_h}+M_{G^{h}}(\beta^h_{d_h})$. We shall prove that $F$ is
irreducible. Let to this end $Y(\underline{t})=\sum_p
Y(p){\underline{t}}^p$ be a root of $g_h(t_1^{n\over
d_h},\ldots,t_e^{n\over d_h},y)=0$ and consider the
$\displaystyle{({m_h\over d_h},Z)}$ deformation
$\tilde{Y}=\sum_{p\in {\displaystyle {1\over d_h}.M_h}}
Y(p){\underline{t}}^p+Zt^{m_h\over d_h}$ of $Y(\underline{t})$.
Let
$M_{G^{h}}(\beta_{d_h}^h)=c.{\underline{x}}^{\underline{\theta}_0}.g_1^{\theta_1}.\ldots.g_{h-1}^{\theta_{h-1}}$,
where $c\in {\bf K}^*$. Since $g_h$ is irreducible, then

$$
O(F,g_h)=O(M_{G^{h}}(\beta_{d_h}^h),g_h)=\sum_{i=1}^e\theta_0^i{r_i^0\over
d_h}+\sum_{k=1}^{h-1}\theta_k{r_k\over d_h}
$$

\noindent  but $g_h(t_1^{n\over d_h},\ldots,t_e^{n\over
d_h},\tilde{Y})=c(Z)t^{r_h\over d_h}$, deg$_Zc(Z) > 0$, and
inco$(g_k(t_1^{n\over d_h},\ldots,t_e^{n\over
d_h},\tilde{Y}))\in{\bf K}^*$ for all $1\leq k\leq h-1$, in
particular
info$(F(t_1^n,\ldots,t_e^{n},\tilde{Y}(t_1^{d_h},\ldots,t_e^{d_h},Z))=\tilde{c}(Z)t^{r_h}$
and deg$_Z(\tilde{c}(Z)) > 0$. This implies that there exists
$z_0\in {\bf K}$ such that if
$y(\underline{t})=\tilde{Y}(t_1^{d_h},\ldots,t_e^{d_h},z_0)$, then
exp$(F(t_1^n,\ldots,t_e^n,y(\underline{t})))
> r_hd_h$. Since $F$ is monic in $y$ and the minimal polynomial of
$y(x_1^{1\over n},\ldots,x_e^{1\over n})$ over ${\bf
K}((x_1,\ldots,x_e))$ is  of degree $n$, then this polynomial
coincides with  $F$, which is consequently irreducible. Now
$O(F,f)=O(F-f,f)> r_hd_h$, then $f$ is irreducible by Proposition
6.3.
\end{proofs}

\section{Examples}

\noindent {\bf Example 1}: Let
$f=y^8-2x_1x_2y^4+x_1^2x_2^2-x_1^3x_2^2\in{\bf K}[[x_1,x_2]][y]$.
Then we have:

- $D_1=n^2=8^2=64, d_1=n=8$, $r_0^1=(8,0),r_0^2=(0,8)$,
$g_1=$App$_{d_1}(f)=y$, and $r_1=O(f,g_1)=(2,2)$.

- $D_2$ is the gcd of the $(2,2)$ minors of the matrix
$(8.I(2,2),(2,2)^T)$, then $D_2=16=8.2$, in particular $d_2=2$.
Since $f=(y^4-x_1x_2)^2-x_1^3x_2 ^2$, then
$g_2$=App$_{d_2}(f)=y^4-x_1x_2$. Let ${\underline{
r}}^2=(\dfrac{r_0^1}{d_2},\dfrac{r_0^2}{d_2},\dfrac{r_1}{d_2})=((4,0),(0,4),(1,1))$
and ${\underline{G}}^2=(g_1)$, then
$r_2=$fO$({\underline{r}}^2,{\underline{G}}^2,x_1^3x_2^2)=3(4,0)+2(0,4)=(12,8)$.

- $D_3$ is the gcd of the $(2,2)$ minors of the matrix
$(8.I(2,2),(2,2)^T,(12,8)^T)$, then $D_3=8$, in particular
$d_3=1$.

- Now GNP$(g_2,{\underline{ r}}^2,{\underline{G}}^2)=\lbrace
((0,0),4.(1,1)),((4,4),(0,0))\rbrace$ and GNP$(f,{\underline{
r}}^3=(r_0^1,r_0^2,r_1,r_2),{\underline{G}}^3=(g_1,g_2))=\lbrace
((0,0),2.(12,8)),((24,16),(0,0))\rbrace$, then the strict
straightness condition is verified. Since $r_1d_1 < r_2d_2$, then
$f$ is irreducible. Note that $m_2=(10,6)$ is the second
characteristic exponent of $f$.

\medskip

\noindent {\bf Example 2}: Let
$f=y^8-2x_1x_2y^4+x_1^2x_2^2-x_1^4x_2^2-x_1^5x_2^3\in{\bf
K}[[x_1,x_2]][y]$. Then we have:

- $D_1=n^2=8^2=64, d_1=n=8$, $r_0^1=(8,0),r_0^2=(0,8)$,
$g_1=$App$_{d_1}(f)=y$, and $r_1=O(f,g_1)=(2,2)$.

- $D_2$ is the gcd of the $(2,2)$ minors of the matrix
$(8.I(2,2),(2,2)^T)$, then $D_2=16=8.2$, in particular $d_2=2$.
Since $f=(y^4-x_1x_2)^2-x_1^4x_2 ^2-x_1^5x_2^3$, then
$g_2$=App$_{d_2}(f)=y^4-x_1x_2$. Let ${\underline{
r}}^2=(\dfrac{r_0^1}{d_2},\dfrac{r_0^2}{d_2},\dfrac{r_1}{d_2})=((4,0),(0,4),(1,1))$
and ${\underline{G}}^2=(g_1)$, then
$r_2=$fO$({\underline{r}}^2,{\underline{G}}^2,x_1^4x_2^2)=4(4,0)+2(0,4)=(16,8)$.

- $D_3$ is the gcd of the $(2,2)$ minors of the matrix
$(8.I(2,2),(2,2)^T,(16,8)^T)$, then $D_3=16$, in particular
$d_3=d_2=2$. In particular $f$ is not irreducible. Note that in
this example the strict straightness condition is verified for $f$
and $g_2$.

\medskip

\noindent {\bf Example 3}: Let
$f=y^8-2x_1x_2y^4+x_1^3x_2^2-x_1y^5\in{\bf K}[[x_1,x_2]][y]$. Then
we have:

- $D_1=n^2=8^2=64, d_1=n=8$, $r_0^1=(8,0),r_0^2=(0,8)$,
$g_1=$App$_{d_1}(f)=y$, and $r_1=O(f,g_1)=(3,2)$.

- $D_2$ is the gcd of the $(2,2)$ minors of the matrix
$(8.I(2,2),(3,2)^T)$, then $D_2=8$, in particular $d_2=1$.

-  GNP$(f,{\underline{
r}}^2=(r_0^1,r_0^2,r_1),{\underline{G}}^2=(g_1))=\lbrace
((0,0),8.(3,2)),((8,0),5.(3,2)),((8,0)+(0,8),4.(3,2)),(3.(8,0)+2.(0,8),(0,0))\rbrace=
\lbrace
((0,0),(24,16)),((8,0),(15,10)),((8,8),(12,8)),((24,16),(0,0))\rbrace$.
Here the strict straightness is not verified, then $f$ is not
irreducible.

\end {document}